\documentclass{ajmart}

\issueinfo{1}{4}{December}{2007}

\copyrightinfo{2007}{Aulona Press  \emph{(Albanian J. Math.)}}

\def\issn{{\sc ISSN} 1930-1235: }
\def\issueyear{2007}

\PII{\issn  (\issueyear )}

\pagespan{253}{270} 

\usepackage{amsmath,amsbsy,amsfonts,amsopn,amstext,euscript,amssymb,amsthm,latexsym,amsxtra,euscript,amscd}

\newtheorem{thm}{Theorem}
\newtheorem{prop}[thm]{Proposition}
\newtheorem{lem}[thm]{Lemma}
\newtheorem{rem}[thm]{Remark}
\newtheorem{cor}[thm]{Corollary}

\theoremstyle{definition}

\theoremstyle{remark}

\newcommand{\ch}[2]
{\begin{bmatrix}
 #1 \\
 #2\\
\end{bmatrix}}

\newcommand{\chr}[4]
{\begin{bmatrix}
 #1 & #2\\
 #3 & #4
\end{bmatrix}}

\newcommand{\chs}[6]
{\begin{bmatrix}
 #1 & #2 & #3\\
 #4 & #5 & #6
\end{bmatrix}}

\def\Z{\mathbb Z}
\def\Q{\mathbb Q}
\def\R{\mathbb R}
\def\C{\mathbb C}
\def\P{\mathbb P}

\def\H{\mathcal H}
\def\M{{\mathcal M}}
\def\A{\mathcal A}
\def\X{\mathcal X}

\def\L{\mathcal L}

\def\m{\mathfrak m}
\def\an{\mathfrak a}
\def\bn{\mathfrak b}
\def\en{\mathfrak e}
\def\hn{\mathfrak h}

\def\xs{{\Bbb o}}

\def\J{\mbox{Jac }}

\def\<{\langle}
\def\>{\rangle}

\def\a{\alpha}
\def\b{\beta}
\def\t{\tau}
\def\s{\sigma}

\def\T{\theta}
\def\d{{\delta }}

\def\O{\Omega}
\def\om{\omega}

\def\e{\eta}
\def\G{\Gamma}
\def\l{\lambda}

\def\iso{{\, \cong\, }}

\def\embd{\hookrightarrow}
\def\Aut{\mbox{Aut}}

\title{Thetanulls of cyclic curves of small genus}

\author{E. Previato, T. Shaska, \and G. S. Wijesiri}

\thanks{$\star$ The first author was supported in part by the NSF grant NSF-DMS-0205643}

\thanks{$\star \star$  The second author was partially supported by an NSF grant and by the NATO grant ICS.EAP.ASI. No. 982903. }



\begin{document}

\begin{abstract}
We study relations among the classical thetanulls of cyclic curves, namely curves $\X$ (of genus $g(\X )>1$ )
with an automorphism $\s$  such that $\s$ generates a normal subgroup of the group $G$ of automorphisms, and
$g \left( \X/ \<\s\> \right) =0$. Relations between thetanulls and branch points of the projection are the
object of much classical work, especially for hyperelliptic curves, and of recent work, in the cyclic case.
We   determine  the curves of genus 2 and 3 in the locus $\mathcal M_g (G, \textbf{C})$ for all $G$ that have
a normal subgroup $\langle\s\rangle$ as above, and all possible signatures \textbf{C}, via relations among
their thetanulls.
\end{abstract}

\maketitle

\section{Introduction}

In this paper we consider cyclic algebraic curves, over the complex numbers. These are by definition compact
Riemann surfaces $\X$ of genus $g>1$ (unless we allow singular points, as noted below, so as not attach
unnecessary qualifications to a definition or statement), admitting an automorphism $\sigma$ such that $\X
/\sigma\cong\P^1$ and $\sigma$ generates a normal subgroup of the automorphism group $Aut(\X )$ of $\X$. When
the curve is hyperelliptic, we insist that the curve have ``extra automorphisms'', in particular $\sigma$ is
not the hyperelliptic involution. Note that the condition  implies to having an equation $y^n=f(x)$ for the
curve, where $x$ is an affine coordinate on $\P^1$, $\sigma$ has order $n$, and $1, y, \sigma y,...,
\sigma^{n-1}y$ is a basis of $\C(\X )/\C (x).$  Naturally, the branch points  of \, $\pi : \X \rightarrow
\P^1$, together with the signature \textbf{C} of the cover (its monodromy up to conjugation) provide
algebraic coordinates for the curve in moduli, though the same curve could be represented in different ways.
The problem of expressing these algebraic data in terms of the transcendental (period matrix, thetanulls,
e.g.) is classical.  We use below formulas for genus-2 curves due to Rosenhein and Picard, Thomae's formulas
for hyperelliptic curves, and a recent generalization of the latter for cyclic curves with $\<s\> \cong C_3$,
where we denote by $C_n$ the cyclic group of order $n$, due to Nakayashiki \cite{NK}; several other authors
recently obtained partial generalizations to cyclic curves also. We do not aim here at a complete account of
the classical or contemporary work on these problems.

Cyclic curves are rare in the moduli space $\M_g$ of smooth curves, and it is desirable to characterize their
locus, by algebraic conditions on the equation of the curve, or by analytic conditions on its Abelian
coordinates, in other words, theta functions, and better yet, by both.  We achieve this for genera 2 and 3,
making recourse to classical formulas, some recent results of Hurwitz space theory, and symbolic
manipulation.

The contents of the paper are as follows. In section 2 we recall the notation for Riemann's theta function,
as well as classical facts on theta characteristics; we recall Frobenius' and Thomae's formulas for
hyperelliptic curves. In sections 3 and 4, respectively, we specialize to the case of genera 2 and 3, we
recall recent results on $\mathcal M_g (G, \textbf{C})$, and we calculate thetanull constraints that define
the loci of the cyclic curves, using the results we cited.  The cleanest case is the one of genus 2 and
$\langle\sigma\rangle\cong C_2$, which was classified by Jacobi who gave a condition in terms the branch
points of the hyperelliptic involution; such a condition was extended, in principle, to any curve in
$\mathcal M_g (C_n, \textbf{C})$, cf. \cite{kuhn} or \cite{Sh7}, but the algebraic equation satisfied by the
branch points would rapidly become intractable with the size of $n$.

\section{Preliminaries} In this section we give a brief description of the basic setup. All of this material
can be found in any standard book on theta functions.

Let $\X$ be a genus $g \geq 2$ algebraic curve. We choose a symplectic homology basis for $\X$, say $ \{ A_1,
\dots, A_g, B_1, \dots , B_g\},$ such that the intersection products $A_i \cdot A_j = B_i \cdot B_j =0$ and
$A_i \cdot B_j= \d_{i j},$ where $\d_{i j}$ is the Kronecker delta. We choose a basis $\{ w_i\}$ for the
space of holomorphic 1-forms such that $\int_{A_i} w_j = \d_{i j}$. The matrix $\O= \left[ \int_{B_i} w_j
\right] $ is  the \emph{period matrix} of $\X$.  The columns of the matrix $\left[ I \ | \O \right]$ form a
lattice $L$ in  $\C^g$ and the Jacobian  of $\X$ is $\J (\X) = \C^g/ L$. Let $\H_g$ be the \emph{Siegel
upper-half space}. Then $\O \in \H_g$ and there is an injection
\[ \M_g \embd \H_g/ Sp_{2g}(\Z) =: \A_g \] where $Sp_{2g}(\Z)$ is the \emph{symplectic group}.
For any $z \in \C^g$ and $\t \in \H_g$ \emph{Riemann's theta function} is defined as
\[ \T (z , \t) = \sum_{u\in \Z^g} e^{\pi i ( u^t \t u + 2 u^t z )  } \]
where $u$ and $z$ are $g-$dimensional column vectors and the products involved in the formula are matrix
products. The fact that the imaginary part of $\t$ is positive makes the series absolutely convergent over
any compact sets. Therefore, the function is analytic.  The theta function is holomorphic on $\C^g\times
\H_g$ and satisfies
\[\T(z+u,\tau)=\T(z,\tau),\quad \T(z+u\tau,\tau)=e^{-\pi i( u^t \tau u+2z^t u )}\cdot
\T(z,\tau),\]
where $u\in \Z^g$; see \cite{Mu1} for details. Any point $e \in \J (\X)$ can be written uniquely as
$e = (b,a) \begin{pmatrix} 1_g \\ \O \end{pmatrix}$, where $a,b \in \R^g.$  We shall use the notation $[e] =
\ch{a}{b}$ for the characteristic of $e$. For any $a, b \in \Q^g$, the theta function with rational
characteristics is defined as
\[ \T  \ch{a}{b} (z , \t) = \sum_{u\in \Z^g} e^{\pi i ( (u+a)^t \t (u+a) + 2 (u+a)^t (z+b) )  }. \]
When the entries of column vectors $a$ and $b$ are from the set $\{ 0,\frac{1}{2}\}$, then the
characteristics $ \ch {a}{b} $ are called the \emph{half-integer characteristics}. The corresponding theta
functions with rational characteristics are called \emph{theta characteristics}. A scalar obtained by
evaluating a theta characteristic at $z=0$ is called a \emph{theta constant}. Points of order $n$ on $\J _\X$
are called the $\frac 1 n$-\emph{periods}. Any half-integer characteristic is given by
\[
\m = \frac{1}{2}m = \frac{1}{2}
\begin{pmatrix} m_1 & m_2 &  \cdots &  m_g \\ m_1^{\prime} & m_2^{\prime} & \cdots & m_g^{\prime}  \end{pmatrix}
\]
where $m_i, m_i^{\prime} \in \Z.$  For $\gamma = \ch{\gamma ^\prime}{\gamma^{\prime \prime}} \in
\frac{1}{2}\Z^{2g}/\Z^{2g}$ we define $e_*(\gamma) = (-1)^{4 (\gamma ^\prime)^t \gamma^{\prime \prime}}.$
Then,
\[ \T [\gamma] (-z , \t) = e_* (\gamma) \T [\gamma] (z , \t).\]
We say that $\gamma$ is an  \textbf{even}  (resp.  \textbf{odd}) characteristic if $e_*(\gamma) = 1$ (resp.
$e_*(\gamma) = -1$). For any curve of genus $g$, there are $2^{g-1}(2^g+1)$ (respectively $2^{g-1}(2^g-1)$ )
even theta functions (respectively odd theta functions). Let $\an$ be another half integer characteristic. We
define $\m\, \an$ as follows.
\[
%
 \m \, \an = \frac{1}{2} \begin{pmatrix} t_1 & t_2 &  \cdots &  t_g \\ t_1^{\prime} & t_2^{\prime} & \cdots &
t_g^{\prime}
\end{pmatrix}
\]
where $t_i \equiv (m_i\, + a_i)  \mod 2$ and $t_i^{\prime} \equiv (m_i^{\prime}\, + a_i^{\prime} ) \mod 2.$

For the rest of this section we consider only characteristics $\frac{1}{2}q$ in which each of the elements
$q_i,q_i^{\prime}$ is either 0 or 1. We use the following abbreviations
\[
\begin{split}
&|\m| = \sum_{i=1}^g m_i m_i^{\prime},  \quad \quad \quad \quad \quad \quad \quad \quad \quad
|\m, \an| = \sum_{i=1}^g (m_i^{\prime} a_i - m_i a_i^{\prime}), \\
& |\m, \an, \bn| = |\an, \bn| + |\bn, \m| + |\m, \an|, \quad \quad {\m\choose \an} = e^{\pi i \sum_{j=1}^g
m_j a_j^{\prime}}.
\end{split}
\]

The set of all half integer characteristics forms a group $\G$ which has $2^{2g}$ elements. We say that two
half integer characteristics $\m$ and $\an$ are \emph{syzygetic} (resp., \emph{azygetic}) if $|\m, \an|
\equiv 0 \mod 2$ (resp., $|\m, \an| \equiv 1 \mod 2$) and three half integer characteristics $\m, \an$, and
$\bn$ are syzygetic if $|\m, \an, \bn| \equiv 0 \mod 2$.

A \emph{G\"opel group} $G$ is a group of $2^r$ half integer characteristics where $r \leq g$ such that every
two characteristics are syzygetic. The elements of the group $G$ are formed by the sums of $r$ fundamental
characteristics; see \cite[pg. 489]{Baker} for details. Obviously, a G\"opel group of order $2^r$ is
isomorphic to $C^r_2$. The proof of the following lemma can be found on   \cite[pg.  490]{Baker}.
\begin{lem}
The number of different G\"opel groups which have $2^r$ characteristics is
\[
\frac{(2^{2g}-1)(2^{2g-2}-1)\cdots(2^{2g-2r+2}-1)}{(2^r-1)(2^{r-1}-1)\cdots(2-1)}
\]
\end{lem}
If $G$ is a G\"opel group with $2^r$ elements, then it has $2^{2g-r}$ cosets. The cosets are called
\emph{G\"opel systems} and denoted by $\an G$, $\an \in \G$. Any three characteristics of a G\"opel system
are syzygetic. We can find a set of characteristics called a basis of the G\"opel system which derives all
its $2^r$ characteristics by taking only the combinations of any odd number of characteristics of the basis.

\begin{lem}
Let $g \geq 1$ be a fixed integer, $r$ be as defined above and $\sigma = g-r.$ Then there are
$2^{\sigma-1}(2^\sigma+1)$ G\"opel systems which consist of even characteristics only and there are
$2^{\sigma-1}(2^\sigma-1)$ G\"opel systems which consist of odd characteristics. The other
$2^{2\sigma}(2^r-1)$ G\"opel systems consist as many odd characteristics as even characteristics.
\end{lem}

\proof The proof can be found on \cite[pg. 492]{Baker}. \qed

\begin{cor}\label{numb_systems}
When $r=g$ we have only one (resp., 0) G\"opel system which consists of even (resp., odd) characteristics.
\end{cor}

\begin{prop}\label{ThetaId} The following statements are true.
\begin{equation}\label {eq1}
\T^2[\an] \T^2[\an \hn] = \frac{1}{2^{g-1}} \sum_\en  e^{\pi i |\an \en|} { \hn \choose \an \en}
\T^2[\en]\T^2[\en \hn]
\end{equation}
\begin{equation}\label {eq2}
\T^4[\an] + e^{\pi i |\an, \hn|} \T^4[\an \hn] = \frac{1}{2^{g-1}} \sum_\en  e^{\pi i |\an \en|} \{ \T^4[\en]
+ e^{ \pi i |\an, \hn|} \T^4[\en \hn]\}
\end{equation}
where $\T[e]$ is the theta constant corresponding to the characteristic $e,$  $\an $ and $ \hn $ are any half
integer characteristics and $\en$ is an even characteristic such that $|\en| \equiv |\en \hn| \mod 2$. There
are $2\cdot 2^{g-2}\, (2^{g-1} +1) $ such candidates for $\en.$
\end{prop}
\begin{proof}
For the proof, see  \cite[pg.  524]{Baker}.
\end{proof}
The statements given in the proposition above can be used to get identities among theta constants; see
section 3.

\subsection{Cyclic curves with extra automorphisms}
A normal cyclic curve is an algebraic curve $\X$ such that there exist a normal cyclic subgroup $C_m
\triangleleft \,\Aut(\X)$ such that $g( \X/{C_m}) = 0.$  Then $\bar{G} = G/C_m$ embeds as a finite subgroup
of $PGL(2,\C)$.  An affine  equation of a birational model of a cyclic curve can be given by the following
\begin{equation}\label{cyclic}
y^m = f(x) = \prod_{i=1}^s (x-\a_i)^{d_i} , \,  0 < d_i < m.
\end{equation}

Hyperelliptic curves are cyclic curves with $m=2$. Note that when $0<d_i$ for some $i$ the curve is singular.
A hyperelliptic curve $\X$ is a cover of order two of the projective line $\P^1.$ Let $z$ be the generator
(the hyperelliptic involution) of the Galois group $Gal(\X / \P^1).$ It is known that $\langle z \rangle $ is
a normal subgroup of the automorphism group $\Aut(\X)$.  Let $\X \longrightarrow \P^1$ be the degree 2
hyperelliptic projection. We can assume that infinity is  a branch point. Let
\[ B := \{\a_1,\a_2, \cdots ,\a_{2g+1} \}\]
be the set of other branch points.  Let $S = \{1,2, \cdots, 2g+1\}$ be the index set of $B$ and $\e : S
\longrightarrow \frac{1}{2}\Z^{2g}/\Z^{2g}$  be a map defined as follows;
\[
\begin{split}
\e(2i-1) & = \begin{bmatrix}
              0 & \cdots & 0 & \frac{1}{2} & 0 & \cdots & 0\\
              \frac{1}{2} & \cdots & \frac{1}{2} & 0 & 0 & \cdots & 0\\
            \end{bmatrix} \\
 \e(2i) & =\begin{bmatrix}
              0 & \cdots & 0 & \frac{1}{2} & 0 & \cdots & 0\\
              \frac{1}{2} & \cdots & \frac{1}{2} & \frac{1}{2} & 0 & \cdots & 0\\
            \end{bmatrix}
\end{split}
\]
where the nonzero element of the first row appears in $i^{th}$ column. We define  $\e(\infty) $ to be $
\begin{bmatrix}
              0 & \cdots & 0 & 0\\
              0 & \cdots & 0 & 0\\
            \end{bmatrix}$.
For any $T \subset B $, we can define the half-integer characteristic as
\[ \e_T = \sum_{a_k \in T } \e(k) .\]

Let $T^c$ denote the complement of $T$ in $B.$ Note that $\e_B \in \Z^{2g}.$ If we view $\e_T$ as an element
of $\frac{1}{2}\Z^{2g}/\Z^{2g}$ then $\e_T= \e_{T^c}.$ Let $\vartriangle$ denote the symmetric difference of
sets, that is $T \vartriangle R = (T \cup R) - (T \cap R).$ It can be shown that the set of subsets of $B$ is
a group under $\vartriangle$. We have the following group isomorphism
\[ \{T \subset B\,  |\, \#T \equiv g+1 \mod 2\} / T \cong \frac{1}{2}\Z^{2g}/\Z^{2g}.\]

For hyperelliptic curves, it is known that $2^{g-1}(2^g+1) - {2g+1 \choose g}$ of the even theta constants
are zero. The following theorem provides a condition on the characteristics in which theta characteristics
become zero. The proof of the theorem can be found  in \cite[pg.  102]{Mu2}.
\begin{thm}\label{vanishingProperty}
Let $\X$ be a hyperelliptic curve, with a set $B$ of branch points. Let $S$ be the index set as above and $U
$ be the set of all odd values of $S$. Then for all $T \subset S$ with even cardinality, we have $ \T[\e_T] =
0$  if and only if $\#(T \triangle U) \neq g+1$, where $\T[\e_T]$ is the theta constant corresponding to the
characteristics $\e_T$.
\end{thm}
Notice also that by parity, all odd theta constants are zero. There is a formula (so called Frobenius' theta
formula) which half-integer theta characteristics for hyperelliptic curves satisfy.

\begin{lem} [Frobenius]
For all $z_i \in \C^g$, $1\leq i \leq 4$ such that $z_1 + z_2 + z_3 + z_4 = 0$ and for all $b_i \in \Q^{2g}$,
$1\leq i \leq 4$ such that $b_1 + b_2 + b_3 + b_4 = 0$, we have
\[ \sum_{j \in S \cup \{\infty\}} \epsilon_U(j) \prod_{i =1}^4 \T[b_i+\e(j)](z_i) = 0, \]
where for any $A \subset B$,
\[
\epsilon_A(k) =
          \begin{cases}
           1 & \textit {if $k \in A$} \\
           -1 & \textit {otherwise}
          \end{cases}
\]
\end{lem}

\proof See \cite[pg.  107]{Mu1}.

\qed

A relationship between theta constants and the branch points of the hyperelliptic curve is given by Thomae's
formula.
\begin{lem}[Thomae]\label{Thomae}
For a non singular even half integer characteristics $e$ corresponding to the partition of the branch points
$\{1,2,\cdots,2(g+1)\} = \{i_1<i_2<\cdots<i_{g+1} \} \cup \{j_1<j_2<\cdots<j_{g+1} \},$ we have
\[\T[e](0;\t)^8 = A \, \prod_{k<l} (\l_{i_k} - \l_{i_l})^2 (\l_{j_k} - \l_{j_l})^2. \]
\end{lem}

See   \cite[pg.  128]{Mu1} for the description of $A$ and \cite[pg.  120]{Mu1} for the proof. Using Thomae's
formula and Frobenius' theta identities we express the branch points of the hyperelliptic curves in terms of
even theta constants.

\section{Genus 2 curves}
The  automorphism group $G$ of a genus 2 curve $\X$ in characteristic $\ne2$ is isomorphic to \ $\Z_2$,
$\Z_{10}$, $V_4$, $D_8$, $D_{12}$, $SL_2 (3)$, $ GL_2(3)$, or $2^+S_5$. The case when $G \iso 2^+S_5$ occurs
only in characteristic 5. If $G \iso SL_2 (3)$ (resp., $ GL_2(3)$) then $\X$ has equation $Y^2=X^6-1$ (resp.,
$Y^2=X(X^4-1)$). If $G \iso \Z_{10}$ then $\X$ has equation $Y^2=X^6-X$.
For a fixed $G$ from the list above, the locus of genus 2 curves with automorphism group $G$ is an irreducible
algebraic subvariety of $\M_2$. Such loci can be described in terms of the Igusa invariants.

For any genus 2 curve we have six odd theta characteristics and ten even theta characteristics. The following
are the sixteen theta characteristics, where the first ten are even and the last six are odd. For simplicity,
we denote them by $\T_i = \ch{a} {b}$ instead of $\T_i \ch{a} {b} (z , \t)$ where $i=1,\dots ,10$ for the
even theta functions.
\begin{small}
\[
\begin{split}
\T_1 = \chr {0}{0}{0}{0} , \,  \T_2 = \chr {0}{0}{\frac{1}{2}} {\frac{1}{2}} ,  \T_3 =\chr
{0}{0}{\frac{1}{2}}{0} , \, \, \T_4 = \chr {0}{0}{0}{\frac{1}{2}} , \, \,    \T_5 = \chr{\frac{1}{2}}{0}
{0}{0} ,\\
  \T_6  = \chr {\frac{1}{2}}{0}{0}{\frac{1}{2}} , \, \,
  \T_7 = \chr{0}{\frac{1}{2}} {0}{0} , \, \,
  \T_8 = \chr{\frac{1}{2}}{\frac{1}{2}} {0}{0} , \, \,
  \T_9 = \chr{0}{\frac{1}{2}} {\frac{1}{2}}{0} , \, \,
  \T_{10} = \chr{\frac{1}{2}}{\frac{1}{2}} {\frac{1}{2}}{\frac{1}{2}} ,\\
\end{split}
\]
\end{small}
and the odd theta functions  correspond to the following characteristics
\[   \chr{0}{\frac{1}{2}} {0}{\frac{1}{2}} , \,
   \chr{0}{\frac{1}{2}} {\frac{1}{2}}{\frac{1}{2}} , \,
    \chr{\frac{1}{2}}{0} {\frac{1}{2}}{0} , \, \,
    \chr{\frac{1}{2}}{\frac{1}{2}} {\frac{1}{2}}{0} , \,
    \chr{\frac{1}{2}}{0} {\frac{1}{2}}{\frac{1}{2}} , \,
    \chr{\frac{1}{2}}{\frac{1}{2}} {0}{\frac{1}{2}}  \]
Consider the following G\"opel group
$$G = \left\{0 = \chr{0}{0}{0}{0}, \m_1 = \chr {0}{0}{0}{\frac{1}{2}}, \m_2 = \chr {0}{0}{\frac{1}{2}}{0},
\m_1 \m_2 = \chr {0}{0}{\frac{1}{2}} {\frac{1}{2}} \right\}.$$ Then, the corresponding G\"opel systems are
given by:
\[
\begin{split}
G &= \left\{ \chr{0}{0}{0}{0},  \chr {0}{0}{0}{\frac{1}{2}},  \chr {0}{0}{\frac{1}{2}}{0},
 \chr {0}{0}{\frac{1}{2}} {\frac{1}{2}} \right\} \\
\bn_1 G &= \left\{\chr{\frac{1}{2}}{0} {0}{0}, \chr {\frac{1}{2}}{0}{0}{\frac{1}{2}},
 \chr{\frac{1}{2}}{0} {\frac{1}{2}}{0}, \chr{\frac{1}{2}}{0} {\frac{1}{2}}{\frac{1}{2}} \right\} \\
\bn_2 G &= \left\{\chr{0}{\frac{1}{2}} {\frac{1}{2}}{0}, \chr{0}{\frac{1}{2}} {\frac{1}{2}}{\frac{1}{2}},
\chr{0}{\frac{1}{2}} {0}{0},
\chr{0}{\frac{1}{2}} {0}{\frac{1}{2}} \right\} \\
\bn_3 G &= \left\{\chr{\frac{1}{2}}{\frac{1}{2}} {\frac{1}{2}}{0}, \chr{\frac{1}{2}}{\frac{1}{2}}
{\frac{1}{2}}{\frac{1}{2}}, \chr{\frac{1}{2}}{\frac{1}{2}} {0}{0}, \chr{\frac{1}{2}}{\frac{1}{2}}
{0}{\frac{1}{2}} \right\}
\end{split}
\]
Notice that from all four cosets, only $G$ has all even characteristics as noticed in Corollary~\ref{numb_systems}.
Using the Prop.~\ref{ThetaId} we have the following six identities for the above G\"opel group.
\[
\left \{ \begin{array}{lll}
 \T_5^2 \T_6^2 & = & \T_1^2 \T_4^2 - \T_2^2 \T_3^2 \\
 \T_5^4 + \T_6^4 & =& \T_1^4 - \T_2^4 - \T_3^4 + \T_4^4 \\
 \T_7^2 \T_9^2 & = & \T_1^2 \T_3^2 - \T_2^2 \T_4^2 \\
 \T_7^4 + \T_9^4 &= & \T_1^4 - \T_2^4 + \T_3^4 - \T_4^4 \\
 \T_8^2 \T_{10}^2 & = & \T_1^2 \T_2^2 - \T_3^2 \T_4^2 \\
 \T_8^4 + \T_{10}^4 & = & \T_1^4 + \T_2^4 - \T_3^4 - \T_4^4 \\
\end{array}
\right.
\]
These identities express even theta constants in terms of four theta constants. We call them fundamental
theta constants $\T_1, \, \T_2, \, \T_3, \, \T_4$.

Next we  find  the relation between theta characteristics and branch points of a genus two curve.
\begin{lem}[Picard] Let a genus 2 curve be given by
\begin{equation} \label{Rosen2}
Y^2=X(X-1)(X-\lambda)(X-\mu)(X-\nu).
\end{equation} Then, $\lambda, \mu, \nu$   can be written as follows:
%
\begin{equation}\label{Picard}
\l = \frac{\T_1^2\T_3^2}{\T_2^2\T_4^2}, \quad \mu = \frac{\T_3^2\T_8^2}{\T_4^2\T_{10}^2}, \quad \nu =
\frac{\T_1^2\T_8^2}{\T_2^2\T_{10}^2}.
\end{equation}
\end{lem}

\begin{proof}
There are several ways for relating $\lambda, \mu, \nu$ to theta constants, depending on the ordering of the
branch points of the curve. Let $B = \{\nu, \mu,\lambda, 1,0\}$ be the branch points of the curves in this
order and $U = \{\nu, \lambda, 0\}$ be the set of odd branch points. Using Lemma~\ref{Thomae} we have the
following set of equations of theta constants and branch points.
\begin{equation}\label{Thomaeg=2}
\begin{array}{ll}
\T_1^4 = A \, \nu \l (\mu -1) (\nu - \l) &
\T_2^4  = A \, \mu (\mu -1) ( \nu - \l) \\
\T_3^4  = A \, \mu  \l (\mu - \l) (\nu - \l) &
\T_4^4  = A\, \nu (\nu - \l) (\mu - \l) \\
\T_5^4  = A \, \l \mu (\nu - 1) ( \nu - \mu)&
\T_6^4  = A \, (\nu - \mu) (\nu -\l) ( \mu -\l)\\
\T_7^4  =  A \, \mu (\nu -1) ( \l -1) (\nu - \l) &
\T_8^4  = A \, \mu \nu (\nu - \mu) (\l -1) \\
\T_9^4  = A \, \nu ( \mu -1) (\l - 1) (\mu - \l) &
\T_{10}^4  = A \, \l ( \l - 1) (\nu - \mu), \\
\end{array}
\end{equation}
where $A$ is a constant.
Choosing the appropriate equation from the set Eq.~\eqref{Thomaeg=2} we have the following:
\[ \l^2 =\left(\frac{\T_1^2\T_3^2}{\T_2^2\T_4^2}\right)^2 \quad \mu^2 = \left(\frac{\T_3^2\T_8^2}{\T_4^2\T_{10}^2}\right)^2 \quad \nu^2
=\left(\frac{\T_1^2\T_8^2}{\T_2^2\T_{10}^2}\right)^2.
\]
Each value for $(\l, \mu, \nu )$ gives isomorphic genus 2 curves. Hence, we can choose
\[ \l = \frac{\T_1^2\T_3^2}{\T_2^2\T_4^2}, \quad \mu = \frac{\T_3^2\T_8^2}{\T_4^2\T_{10}^2}, \quad \nu =
\frac{\T_1^2\T_8^2}{\T_2^2\T_{10}^2}.\]
This completes the proof.

\end{proof}
%
One of the main goals of this paper is to describe each locus of genus 2 curves with fixed automorphism group
in terms of the fundamental theta constants.
We have the following
\begin{cor}\label{possibleCurve}
Every genus two curve can be written in the form:
\[
y^2 = x \, (x-1) \, \left(x - \frac {\T_1^2 \T_3^2} {\T_2^2  \T_4^2}\right)\, \left(x^2 \, -   \frac{\T_2^2 \, \T_3^2 +
\T_1^2 \, \T_4^2} { \T_2^2 \, \T_4^2} \cdot    \a  \, x + \frac {\T_1^2 \T_3^2} {\T_2^2 \T_4^2} \, \a^2 \right),
\]
where $\a = \frac {\T_8^2} {\T_{10}^2}$ and in terms of $\, \, \T_1, \dots , \T_4$ is given by
\[
  \a^2 + \frac {\T_1^4 + \T_2^4 - \T_3^4 - \T_4^4}{\T_1^2 \T_2^2 - \T_3^2 \T_4^2 } \, \a + 1 =0
\]
Furthermore,  if $\alpha = {\pm} 1$ then $V_4 \embd Aut(\X)$.
\end{cor}
\proof  Let's write the genus 2 curve in the following form: $$Y^2 = X (X-1) (X-\l) (X-\mu) (X-\nu)$$ where
$\l ,\mu ,\nu$ are given by Eq. \eqref{Picard}. Let $\a := \frac {\T_8^2} {\T_{10}^2}$.  Then,
\[
\begin{array}{ll}
\mu =  \frac{\T_3^2}{\T_4^2}\, \a,  &  \nu =  \frac{\T_1^2}{\T_2^2} \, \a
\end{array}
\]
Using the following two identities,
\begin{equation}\label{Frobenius}
\begin{split}
\T_8^4 + \T_{10}^4 &= \T_1^4+\T_2^4-\T_3^4-\T_4^4\\
\T_8^2 \T_{10}^2 &= \T_1^2 \T_2^2 - \T_3^2 \T_4^2
\end{split}
\end{equation}
%
we have,
\begin{equation}\label{rootof}
 \a^2 +    \frac {\T_1^4 + \T_2^4 - \T_3^4 - \T_4^4}{\T_1^2 \T_2^2 - \T_3^2 \T_4^2 } \,           \a + 1 = 0
\end{equation}
If $\a=\pm 1$ the $\mu \nu = \l$.    It is well known that this implies that the genus 2 curve has an
elliptic involution. Hence,  $V_4 \embd Aut(\X)$.
\endproof
\begin{rem} i) From the above we have that $\T_8^4=\T_{10}^4$ implies that $V_4 \embd Aut(\X)$. Lemma~15
determines a necessary and equivalent statement when $V_4 \embd Aut(\X)$.

ii)  The last part of the lemma above shows that if $\T_8^4=\T_{10}^4$ then all coefficients of the genus 2
curve are given as rational functions of the 4 fundamental theta functions. Such fundamental theta functions
determine the field of moduli of the given curve. Hence, the curve is defined over its field of moduli.
\end{rem}

\begin{cor}
Let $\X$ be a genus 2 curve which has an elliptic involution. Then $\X$ is defined over its field of moduli.
\end{cor}
\noindent This was the main result of  \cite{Ca}.

\subsection{Describing the locus of genus two curves with fixed automorphism group by theta constants}
The locus $\L_2$ of genus 2 curves $\X$ which have an elliptic involution is a closed subvariety of $\mathcal M_2$. Let
$W= \{\a_1, \a_2, \b_1, \b_2, \gamma_1, \gamma_2 \}$ be the set of roots of the binary sextic and $A$ and $B$ be
subsets of $W$ such that $W=A \cup B$ and $| A \cap B | =2$. We define the cross ratio of the two pairs $z_1,z_2 ;
z_3,z_4 $ by
$$(z_1,z_2 ; z_3,z_4) = \frac{z_1 ; z_3,z_4}{z_2 ; z_3,z_4} = \frac{z_1-z_3}{z_1-z_4} : \frac{z_2-z_3}{z_2-z_4}.$$
Take $ A = \{ \a_1, \a_2, \b_1 , \b_2\}$ and $B = \{ \gamma_1, \gamma_2, \b_1, \b_2\}$.
Jacobi \cite{Krazer} gives a   description of $\L_2$ in terms of the cross ratios of the elements of $W.$
$$ \frac {\a_1-\b_1} {\a_1-\b_2} : \frac {\a_2-\b_1} {\a_2-\b_2}= \frac {\gamma_1-\b_1} {\gamma_1-\b_2} : \frac
{\gamma_2-\b_1} {\gamma_2-\b_2}
$$
We recall that the following identities hold for cross ratios:
\[
(\a_1,\a_2\,;\b_1,\b_2)=(\a_2,\a_1;\b_2,\b_1)=(\b_1,\b_2;\a_1,\a_2)=(\b_2,\b_1;\a_2,\a_1)
\]
and
\[
(\a_1,\a_2;\infty,\b_2)=(\infty,\b_2;\a_1,\a_2)=(\b_2;\a_2,\a_1)
\]
Next we want to use this result to determine relations among theta functions for a genus 2 curve in the locus
$\L_2$. Let $\X$ be any genus 2 curve given by equation
$$Y^2=X(X-1)(X-a_1)(X-a_2)(X-a_3)$$
We take $\infty \in A \cap B$. Then there are five cases for $\a \in A \cap B $, where $\a $ is an element of
the set$
 \{0,1, a_1, a_2, a_3\}$. For each of these cases there are three possible relationships for cross ratios as described below:\\
i) $A \cap B = \{ 0, \infty\}$: The possible cross ratios are
$$(a_1,1;\infty,0) = (a_3,a_2;\infty,0)$$  $$(a_2,1;\infty,0) = (a_1,a_3;\infty,0)$$ $$(a_1,1;\infty,0) = (a_2,a_3;\infty,0)$$
ii) $A \cap B = \{ 1, \infty\}$: The possible cross ratios are
$$(a_1,0;\infty,1)=(a_2,a_3;\infty,1)$$  $$(a_1,0;\infty,1)=(a_3,a_2;\infty,1)$$  $$(a_2,0;\infty,1)=(a_1,a_3;\infty,1)$$
iii) $A \cap B = \{ a_1, \infty\}$: The possible cross ratios are
 $$(1,0;\infty,a_1)=(a_3,a_2;\infty,a_1)$$  $$(a_2,0;\infty,a_1)=(1,a_3;\infty,a_1)$$  $$(1,0;\infty,a_1)=(a_2,a_3;\infty,a_1)$$
iv) $A \cap B= \{ a_2, \infty\}$: The possible cross ratios are
 $$(1,0;\infty,a_2)=(a_1,a_3;\infty,a_2)$$  $$(1,0;\infty,a_2)=(a_3,a_1;\infty,a_2)$$  $$(a_1,0;\infty,a_2)=(1,a_3;\infty,a_2)$$
v) $A \cap B = \{ a_3, \infty\}$: The possible cross ratios are
 $$(a_1,0;\infty,a_3)=(1,a_2;\infty,a_3)$$ $$(1,0;\infty,a_3)=(a_2,a_1;\infty, a_3)$$ $$(1,0;\infty,a_3)=(a_1,a_2;\infty,a_3)$$
We summarize these relationships in the following table:
\begin{tiny}
\begin{center}
\begin{table}[ht!]\label{tab_1}
\begin{tabular}{|c|c|c|c|}
\hline & Cross ratio &  $f(a_1,a_2,a_3)=0$ & theta constants \tabularnewline[8pt]
\hline
1& $(1,0;\infty,a_1)=(a_3,a_2;\infty,a_1)$ & $a_1a_2+a_1-a_3a_1-a_2$ &
$-\T_1^2\T_3^2\T_8^2\T_2^2-\T_1^2\T_2^2\T_4^2\T_{10}^2+$\\

& & &$\T_1^4\T_3^2\T_{10}^2+ \T_3^2\T_2^4\T_{10}^2$ \tabularnewline[4pt]
\hline
2 & $(a_2,0;\infty,a_1)=(1,a_3;\infty,a_1)$ & $a_1a_2-a_1+a_3a_1-a_3a_2$ &
$\T_3^2\T_8^2\T_2^2\T_4^2-\T_2^2\T_4^4\T_{10}^2+$\\
& && $\T_1^2\T_3^2\T_4^2\T_{10}^2-\T_3^4\T_2^2\T_{10}^2$ \tabularnewline[4pt]
\hline
3& $(1,0;\infty,a_1)=(a_2,a_3;\infty,a_1)$ & $a_1a_2-a_1-a_3a_1+a_3$ & $-\T_8^4\T_3^2\T_2^2+\T_8^2\T_2^2\T_{10}^2\T_4^2+$ \\
& &&$\T_1^2\T_3^2\T_8^2\T_{10}^2-\T_3^2\T_2^2\T_{10}^4$ \tabularnewline[4pt]
\hline
4& $(1,0;\infty,a_2)=(a_1,a_3;\infty,a_2)$ & $a_1a_2-a_2-a_3a_2+a_3$ & $-\T_1^2\T_8^4\T_4^2-\T_1^2\T_{10}^4\T_4^2+$ \\
& &&$\T_8^2\T_2^2\T_{10}^2\T_4^2+\T_1^2\T_3^2\T_8^2\T_{10}^2$ \tabularnewline[4pt]
\hline
5& $(1,0;\infty,a_2)=(a_3,a_1;\infty,a_2)$ & $a_1a_2-a_1+a_2-a_3a_2$&$-\T_1^2\T_8^2\T_3^2\T_4^2+\T_1^2\T_{10}^2\T_4^4+$\\
&&&$\T_1^2\T_3^4\T_{10}^2-\T_3^2\T_2^2\T_{10}^2\T_4^2$ \tabularnewline[4pt]
\hline
6 & $(a_1,0;\infty,a_2)=(1,a_3;\infty,a_2)$ & $a_1a_2-a_3a_1-a_2+a_3 a_2$ &
$-\T_1^2\T_8^2\T_2^2\T_4^2+\T_1^4\T_{10}^2\T_4^2-$ \\
& &&$\T_1^2\T_3^2\T_2^2\T_{10}^2+\T_2^4\T_4^2\T_{10}^2$ \tabularnewline[4pt]
\hline
7&$(a_1,0;\infty,a_3)=(1,a_2;\infty,a_3)$ & $a_1a_2-a_3a_1-a_3a_2+a_3$ &
$-\T_8^4\T_2^2\T_4^2+\T_1^2\T_8^2\T_{10}^2\T_4^2-$ \\
&&&$\T_2^2\T_{10}^4\T_4^2+\T_3^2\T_8^2\T_2^2\T_{10}^2$ \tabularnewline[4pt]
\hline
8&$(1,0;\infty,a_3)=(a_2,a_1;\infty, a_3)$&$a_3a_1-a_1-a_3a_2+a_3$ & $\T_8^4-\T_{10}^4$\tabularnewline[4pt]
\hline
9&$(1,0;\infty,a_3)=(a_1,a_2;\infty,a_3)$ &$ a_3a_1+a_2-a_3-a_3a_2$ &
$\T_1^4\T_8^2\T_4^2-\T_1^2\T_2^2\T_4^2\T_{10}^2-$ \\
&&&$\T_1^2\T_3^2\T_8^2\T_2^2+\T_8^2\T_2^4\T_4^2$ \tabularnewline[4pt]
\hline
10&$(a_1,0;\infty,1)=(a_2,a_3;\infty,1)$ & $-a_1+a_3a_1+a_2-a_3 $&$\T_1^4\T_3^2\T_8^2-\T_1^2\T_8^2\T_2^2\T_4^2-$\\
&&&$\T_1^2\T_3^2\T_2^2\T_{10}^2+\T_3^2\T_8^2\T_2^4$ \tabularnewline[4pt]
\hline
11& $(a_1,0;\infty,1)=(a_3,a_2;\infty,1)$ &$ a_1a_2-a_1-a_2+a_3$&$\T_1^2\T_8^4\T_3^2-\T_1^2\T_8^2\T_{10}^2\T_4^2+$\\
&&&$\T_1^2\T_3^2\T_{10}^4-\T_3^2\T_8^2\T_2^2\T_{10}^2$ \tabularnewline[4pt]
\hline
12& $(a_2,0;\infty,1)=(a_1,a_3;\infty,1)$&$a_1-a_2+a_3a_2-a_3$&$\T_1^2\T_8^2\T_4^4-\T_1^2\T_3^2\T_4^2\T_{10}^2+$ \\
&&&$\T_1^2\T_3^4\T_8^2-\T_3^2\T_8^2\T_2^2\T_4^2$ \tabularnewline[4pt]
\hline
13&$(a_1,1;\infty,0) = (a_3,a_2;\infty,0)$ & $a_1a_2-a_3$ & $\T_8^4-\T_{10}^4 $ \tabularnewline[4pt]
\hline
14& $(a_2,1;\infty,0) = (a_1,a_3;\infty,0)$ & $a_1-a_3a_2$ &$ \T_3^4-\T_4^4$ \tabularnewline[4pt]
\hline
15&$(a_1,1;\infty,0) = (a_2,a_3;\infty,0)$ & $a_3a_1-a_2$ & $ \T_1^4-\T_2^4  $ \tabularnewline[4pt] \hline
\end{tabular}
\vspace{1cm} \caption{Relation of theta functions and cross ratios}
\end{table}
\end{center}
\end{tiny}
\begin{lem}\label{lem1}
Let $\X$ be a genus 2 curve. Then  $Aut(\X)\iso V_4$ if and only if the theta functions of $\X$ satisfy
\begin{scriptsize}
\begin{equation}\label{V_4locus1}
\begin{split}
(\T_1^4-\T_2^4)(\T_3^4-\T_4^4)(\T_8^4-\T_{10}^4)
(-\T_1^2\T_3^2\T_8^2\T_2^2-\T_1^2\T_2^2\T_4^2\T_{10}^2+\T_1^4\T_3^2\T_{10}^2+ \T_3^2\T_2^4\T_{10}^2)\\
(\T_3^2\T_8^2\T_2^2\T_4^2-\T_2^2\T_4^4\T_{10}^2+\T_1^2\T_3^2\T_4^2\T_{10}^2-\T_3^4\T_2^2\T_{10}^2) (-\T_8^4\T_3^2\T_2^2+\T_8^2\T_2^2\T_{10}^2\T_4^2+\T_1^2\T_3^2\T_8^2\T_{10}^2-\T_3^2\T_2^2\T_{10}^4)\\
(-\T_1^2\T_8^4\T_4^2-\T_1^2\T_{10}^4\T_4^2+\T_8^2\T_2^2\T_{10}^2\T_4^2+\T_1^2\T_3^2\T_8^2\T_{10}^2) (-\T_1^2\T_8^2\T_3^2\T_4^2+\T_1^2\T_{10}^2\T_4^4+\T_1^2\T_3^4\T_{10}^2-\T_3^2\T_2^2\T_{10}^2\T_4^2)\\
(-\T_1^2\T_8^2\T_2^2\T_4^2+\T_1^4\T_{10}^2\T_4^2-
\T_1^2\T_3^2\T_2^2\T_{10}^2+\T_2^4\T_4^2\T_{10}^2) (-\T_8^4\T_2^2\T_4^2+\T_1^2\T_8^2\T_{10}^2\T_4^2-\T_2^2\T_{10}^4\T_4^2+\T_3^2\T_8^2\T_2^2\T_{10}^2)\\
(\T_1^4\T_8^2\T_4^2-\T_1^2\T_2^2\T_4^2\T_{10}^2-\T_1^2\T_3^2\T_8^2\T_2^2+\T_8^2\T_2^4\T_4^2) (\T_1^4\T_3^2\T_8^2-\T_1^2\T_8^2\T_2^2\T_4^2-\T_1^2\T_3^2\T_2^2\T_{10}^2+\T_3^2\T_8^2\T_2^4)\\
(\T_1^2\T_8^4\T_3^2-\T_1^2\T_8^2\T_{10}^2\T_4^2+\T_1^2\T_3^2\T_{10}^4-\T_3^2\T_8^2\T_2^2\T_{10}^2)
(\T_1^2\T_8^2\T_4^4-\T_1^2\T_3^2\T_4^2\T_{10}^2+\T_1^2\T_3^4\T_8^2-\T_3^2\T_8^2\T_2^2\T_4^2)
 & =0
\end{split}
\end{equation}
\end{scriptsize}
\end{lem}
However, we are unable to get a similar result for cases $D_8$ or  $D_{12}$ by this argument. Instead, we
will use the invariants of genus 2 curves and a more computational approach. In the process, we will offer a
different proof of the lemma above.
\begin{lem}

i) The locus $\L_2$ of genus 2 curves $\X$ which have a degree 2 elliptic subcover is a closed subvariety of
$\mathcal M_2$. The equation of $\L_2$ is given by
\begin{scriptsize}
\begin{equation}
\begin{split}\label{eq_L2_J}
8748J_{10}J_2^4J_6^2- 507384000J_{10}^2J_4^2J_2-19245600J_{10}^2J_4J_2^3
-592272J_{10}J_4^4J_2^2 +77436J_{10}J_4^3J_2^4\\
-81J_2^3J_6^4-3499200J_{10}J_2J_6^3+4743360J_{10}J_4^3J_2J_6-870912J_{10}J_4^2J_2^3J_6
+3090960J_{10}J_4J_2^2J_6^2\\
-78J_2^5J_4^5-125971200000J_{10}^3 +384J_4^6J_6+41472J_{10}J_4^5+159J_4^6J_2^3
-236196J_{10}^2J_2^5-80J_4^7J_2\\
 -47952J_2J_4J_6^4+104976000J_{10}^2J_2^2J_6-1728J_4^5J_2^2J_6+6048J_4^4J_2J_6^2
-9331200J_{10}J_4^2J_6^2 \\
+12J_2^6J_4^3J_6+29376J_2^2J_4^2J_6^3-8910J_2^3J_4^3J_6^2-2099520000J_{10}^2J_4J_6
+31104J_6^5-6912J_4^3J_6^34  \\
-J_2^7J_4^4 -5832J_{10}J_2^5J_4J_6  -54J_2^5J_4^2J_6^2 +108J_2^4J_4J_6^3
+972J_{10}J_2^6J_4^2+1332J_2^4J_4^4J_6  = & 0
\end{split}
\end{equation}
\end{scriptsize}
ii) The locus  of genus 2 curves $\X$ with $Aut(\X)\iso D_8$ is given by the equation of $\L_2$  and
%
\begin{equation}
\label{D_8_locus} 1706J_4^2J_2^2+2560J_4^3+27J_4J_2^4-81J_2^3J_6-14880J_2J_4J_ 6+28800J_6^2 =0
\end{equation}
%
\noindent iii) The locus  of genus 2 curves $\X$ with $Aut(\X)\iso D_{12}$ is
%
\begin{equation}
\label{D_12_locus}
\begin{split}
-J_4J_2^4+12J_2^3J_6-52J_4^2J_2^2+80J_4^3+960J_2J_4J_6-3600
J_6^2 &=0\\
864J_{10}J_2^5+3456000J_{10}J_4^2J_2-43200J_{10}J_4J_2^3-
2332800000J_{10}^2-J_4^2J_2^6\\
-768J_4^4J_2^2+48J_4^3J_2^4+4096J_4^5 &=0\\
\end{split}
\end{equation}
%
\end{lem}
%
Our goal is to express each of the above loci in terms of the theta characteristics. We obtain the following
result.
\begin{thm}\label{thm1}
 Let $\X$ be a genus 2 curve. Then the following hold:

i)  $Aut(\X)\iso V_4$ if and only if the relations of theta functions given Eq.~\eqref{V_4locus1} holds.

ii) $Aut(\X)\iso D_8$ if and only if  Eq.~(1) in \cite{Sh} is satisfied.

iii) $Aut(\X)\iso D_{12}$ if and only if Eq.~(2) in \cite{Sh} is satisfied.
\end{thm}
\proof Part i) of the theorem is Lemma~\ref{lem1}. Here we give a somewhat different proof. Assume that $\X$
is a genus 2 curve with equation
$$Y^2=X(X-1)(X-a_1)(X-a_2)(X-a_3)$$
whose classical invariants satisfy Eq.~\eqref{eq_L2_J}. Expressing the classical invariants of $\X$ in terms
of $a_1, a_2, a_3$, substituting them into \eqref{eq_L2_J}, and factoring the resulting equation yields
\begin{small}
\begin{equation}
\begin{split}\label{L2_factored}
(a_1a_2-a_2-a_3a_2+a_3)^2(a_1a_2-a_1+a_3a_1-a_3a_2)^2(a_1a_2-a_3a_1-a_3a_2+a_3)^2\\
(a_3a_1-a_1-a_3a_2+a_3)^2(a_1a_2+a_1-a_3a_1-a_2)^2(a_1a_2-a_1-a_3a_1+a_3)^2\\
(a_3a_1+a_2-a_3-a_3a_2)^2(-a_1+a_3a_1+a_2-a_3)^2(a_1a_2-a_1-a_2+a_3)^2\\
(a_1a_2-a_1+a_2-a_3a_2)^2(a_1-a_2+a_3a_2-a_3)^2(a_1a_2-a_3a_1-a_2+a_3 a_2)^2 \\
(a_1a_2-a_3)^2 (a_1-a_3a_2)^2  (a_3a_1-a_2)^2  =&\, 0
\end{split}
\end{equation}
\end{small}
It is no surprise that we get the 15 factors of Table 1. The relations of theta constants follow from the
table.
ii) Let $\X$ be a genus 2 curve which has an elliptic involution. Then $\X$ is isomorphic to a curve with
equation
$$Y^2=X(X-1)(X-a_1)(X-a_2)(X-a_1 a_2).$$
If $\Aut (\X) \iso D_8$ then the $SL_2 (k)$-invariants of such curve must satisfy Eq.~\eqref{D_8_locus}.
Then, we get the equation in terms of $a_1, a_2$.
By writing the relation $a_3 = a_1 a_2$ in terms of theta constants, we get $\T_4^4 = \T_3^4$. All the
results above lead to part ii) of the theorem.
iii) The proof of this part is similar to part ii).
\endproof
We would like to express the conditions of the previous lemma in terms of the fundamental theta constants
only.
\begin{lem}
Let $\X$ be a genus 2 curve. Then we have the following:
\begin{description}
\item [i)] $V_4 \hookrightarrow Aut(\X)$ if and only if the fundamental theta constants of $\X$ satisfy
\begin{small}
\begin{equation}\label{V_4locus2}
\begin{split}
\left( \theta_{{3}}^4-\theta_{{4}}^4 \right)  \left(\theta_{{1}}^4 -\theta_{{3}}^4 \right) \left(
\theta_{{2}}^4-\theta_{{4}}^4 \right) \left( \theta_{{1}}^4 -\theta_{{4}}^4 \right)  \left(
\theta_{{3}}^4-\theta_{{2}}^4 \right) \left( \theta_{{1}}^4-
\theta_{{2}}^4 \right) \\
\left( -\theta_{{4}}^2+\theta_{{3}}^2+\theta_{{1}}^2-\theta_{{2}}^2 \right)\left(
 \theta_{{4}}^2-\theta_{{3}}^2+\theta_{{1}}^2-\theta_{{2}}^2
 \right)  \left( -\theta_{{4}}^2-\theta_{{3}}^2+\theta_{{2}}^2+\theta_{{
1}}^2 \right)  \left( \theta_{{4}}^2+\theta_{{3}}^2+\theta_{{2}}^2+\theta_ {{1}}^2 \right)\\
\left( {\theta_{{1}}}^{4}{\theta_{{2}}}^{4}+
{\theta_{{3}}}^{4}{\theta_{{2}}}^{4}+{\theta_{{1}}}^{4}{\theta_{{3}}}^{4}-2\,\theta_{{1}}^2\theta_{{2}}^2\theta
_{{3}}^2\theta_{{4}}^2 \right) \left( -{\theta_{{3}}}^{4}{\theta_{{2}}}^{4}-{
\theta_{{2}}}^{4}{\theta_{{4}}}^{4}-{\theta_{{3}}}^{4}{\theta_{{4}}}^{4} +
2\,\theta_{{1}}^2\theta_{{2}}^2\theta_
{{3}}^2\theta_{{4}}^2 \right)\\
 \left( {\theta_{{2}}}^{4}{\theta_{{4}}}^{4}+{\theta_{{1}}}^{4}{\theta _{{2}}}^{4}+{
\theta_{{1}}}^{4}{\theta_{{4}}}^{4}-2\,\theta_{{1}}^2\theta_{{2}}^2\theta_{{3}}^2\theta_{{4}}^2 \right)
\left( {\theta_{{1}}}^{4}{\theta_{{4}}}^{4}+{\theta_{{3}}}^{4}{\theta_{{4}}}^{4}+{\theta_{{1}}}^{4}{\theta_{
{3}}}^{4}
-2\,\theta_{{1 }}^2\theta_{{2}}^2\theta_{{3}}^2\theta_{{4}}^2\right)  = & 0\\
\end{split}
\end{equation}
\end{small}
\item [ii] $D_8 \hookrightarrow Aut(\X)$ if and only if the fundamental theta constants of $\X$ satisfy
Eq.~(3) in \cite{Sh}

\item [iii] $D_6 \hookrightarrow Aut(\X)$ if and only if the fundamental theta constants of $\X$ satisfy
Eq.~(4) in \cite{Sh}
\end{description}
\end{lem}

\proof Notice that Eq.~\eqref{V_4locus1} contains only $\T_1, \T_2, \T_3, \T_4, \T_8$ and $\T_{10}.$ Using
Eq.~\eqref{Frobenius}, we can eliminate $\T_8$ and $\T_{10}$ from Eq.~\eqref{V_4locus1}.
The $J_{10}$ invariant of any genus two curve is  given by the following in terms of theta constants:
\[ J_{10} = \frac{\T_1^{12} \T_3^{12}}{\T_2^{28} \T_4^{28} \T_{10}^{40}} \, (\T_1^2\T_2^2 - \T_3^2 \T_4^2)^{12} (\T_1^2\T_4^2 - \T_2^2
\T_3^2)^{12} (\T_1^2\T_3^2 - \T_2^2 \T_4^2)^{12}.\] Since $J_{10} \neq 0$ we can cancel the factors
$(\T_1^2\T_2^2 - \T_3^2 \T_4^2), (\T_1^2\T_4^2 - \T_2^2 \T_3^2)$ and $(\T_1^2\T_3^2 - \T_2^2 \T_4^2)$ from
the equation of $V_4$ locus. The result follows from Theorem ~\ref{thm1}.  The proof of part ii) and iii) is
similar and we avoid details.
\endproof
\begin{rem}
i) For the other two loci, we can also obtain equations in terms of the fundamental theta constants. However,
such equations are big and we don't display them here.

ii) By using Frobenius's relations we get  \[ J_{10}= \frac {\left( \T_1 \T_3 \right)^{12}} {\left(\T_2
\T_4\right)^{28} \T_{10}^{16}} \left( \T_5 \T_6 \T_7 \T_8 \T_9 \right)^{24}\]
Hence, $\T_i\neq 0$ for $ i = 1, 3, 5, \dots 9$.
\end{rem}
%

\section{Genus 3  cyclic curves}
For genus 3 we have hyperelliptic and non-hyperelliptic algebraic curves. The following table gives all
possible genus 3 cyclic algebraic curves; see \cite{MS} for details. The first 11 cases are for the
hyperelliptic curves and the last  12 cases are for the non-hyperelliptic curves.

\vspace{.3cm}

\begin{small}

\begin{table}[hb]
\label{g3}
\begin{tabular}{||c|c|c|c||}
\hline \hline  & & &\\
 &$\Aut(\X_g)$ & equation  &  Id.\\
\hline \hline  &&&  \\
1 & $\Z_2$ &$ y^2=x(x-1)(x^5+ax^4+bx^3+cx^2+dx+e)$&  $(2,1)$ \\
%
& &  &\\
2 & $\Z_2 \times \Z_2$ &$ y^2= x^8 + a_3 x^6 + a_2 x^4 + a_1 x^2 + 1$ &  $(4,2)$\\
3 & $\Z_4$ &$y^2= x(x^2-1)(x^4+ax^2+b)$&  $(4,1)$\\
4 & $\Z_{14}$ &$ y^2=x^7-1$ &  $(14,2)$ \\
%
& & & \\
5 & $\Z_2^3$ &$ y^2= (x^4+ax^2+1)(x^4+bx^2+1)$   &$(8,5)$\\
6 & $\Z_2 \times D_8$ &$ y^2= x^8+ax^4+1$ &    $(16,11)$ \\
7 & $\Z_2\times \Z_4$ &$ y^2= (x^4-1)(x^4+ax^2+1)$&    $(8,2)$\\
8 & $D_{12}$ &$ y^2= x(x^6+ax^3+1)$   &$(12,4)$\\
9 & $U_6$ &$ y^2= x(x^6-1)$    & $(24,5)$\\
10 & $V_8$ & $ y^2= x^8-1$&   $(32,9)$\\
%
%
& & & \\
11 & $\Z_2 \times S_4$ & $y^2=x^8+14x^2+1$ &   $(48,48)$ \\
&  & & \\
 \hline\hline
& & & \\

12& $V_4$ &$x^4+y^4+ax^2y^2+bx^2+cy^2+1=0$&(4,2) \\

13 & $D_8$  & take\ $b=c$& (8,3)\\

14 & $S_4$ & take\ $a=b=c$ &(24,12) \\

15 & $C_4^2 \xs S_3$  &  \ take \, $a=b=c=0$ \, or\, $y^4=x(x^2-1)$ & (96,64) \\

\hline

16 &  $16$ & $y^4=x(x-1)(x-t)$ & (16,13)        \\

17 & $48$ & $y^4=x^3-1$ & (48,33)              \\

\hline

18 & $C_3$  & $y^3=x(x-1)(x-s)(x-t)$&(3,1)  \\

19 & $C_6$  & take\ $s=1-t$ & (6,2)     \\

20 & $C_9$  & $y^3=x(x^3-1)$ & (9,1)\\

\hline &&&\\

21 & $L_3(2)$ &  $x^3y+y^3z+z^3x=0$ & (168,42) \\

\hline

&&&\\
22 & $S_3$ & $a(x^4+y^4+z^4)+b(x^2y^2+x^2z^2+y^2z^2)+$ & (6,1)\\

& & $c(x^2yz+y^2xz+z^2xy)=0$&\\
\hline &&&\\

23 & $C_2$ & $ x^4+x^2(y^2+az^2) + by^4+cy^3z+dy^2z^2$ & (2,1)\\
& & $ +eyz^3+gz^4=0$, \  \ either $e=1$ or $g=1$ &\\ \\
\hline \hline
\end{tabular}
\vspace{.5cm} \caption{The list of automorphism groups of genus 3 and their equations}
\end{table}
\end{small}

\vspace{-.3cm}

\subsection{Theta functions for hyperelliptic curves}
For genus three hyperelliptic curve we have 28 odd theta characteristics and 36 even theta characteristics. The
following shows the corresponding characteristics for each theta function. The first 36 are for the even functions and
the last 28 are for the odd functions. For simplicity, we denote them by $\T_i = \ch{a} {b}$ instead of $\T_i \ch {a}
{b} (z , \t).$

\begin{small}
\[
\begin{split}
&\T_1 =  \chs {0}{0}{0}{0} 0 0 , \, \,  \T_2 = \chs {\frac{1}{2}} 0 {\frac{1}{2}} {\frac{1}{2}}
{\frac{1}{2}} {\frac{1}{2}},\, \,  \T_3 =  \chs {\frac{1}{2}} {\frac{1}{2}} {\frac{1}{2}} 0 0 0, \, \, \T_4 = \chs 0 0 0 {\frac{1}{2}} 0 0 ,\\
& \T_5 =  \chs {\frac{1}{2}} 0 0 0 {\frac{1}{2}} 0 , \, \, \T_6 =  \chs {\frac{1}{2}} {\frac{1}{2}} 0 0 0
{\frac{1}{2}}, \, \,
\T_7 =  \chs 0 {\frac{1}{2}} {\frac{1}{2}} {\frac{1}{2}} 0 0  , \, \,  \T_8 =  \chs 0 0
{\frac{1}{2}} 0 {\frac{1}{2}} 0 ,\\
&  \T_9 =  \chs 0 0 0 0 0 {\frac{1}{2}} , \, \, \T_{10} =  \chs {\frac{1}{2}} 0 0 0 0 0 , \, \,
  \T_{11} =  \chs {\frac{1}{2}} {\frac{1}{2}} 0 {\frac{1}{2}}
{\frac{1}{2}} 0 , \, \,
\T_{12} =  \chs {\frac{1}{2}} {\frac{1}{2}} {\frac{1}{2}} {\frac{1}{2}} 0 {\frac{1}{2}} \\
&\T_{13} =  \chs 0 0 0 {\frac{1}{2}} {\frac{1}{2}} 0 , \, \,  \T_{14} =  \chs 0 {\frac{1}{2}} 0 0 0 0, \, \,
\T_{15} =  \chs 0{\frac{1}{2}} {\frac{1}{2}} 0 {\frac{1}{2}} {\frac{1}{2}}, \, \,
\T_{16} =  \chs 0 {\frac{1}{2}} 0 {\frac{1}{2}} 0 {\frac{1}{2}}, \\
& \T_{17} =  \chs 0 0 0 0 {\frac{1}{2}} {\frac{1}{2}} , \, \, \T_{18} =  \chs 0 0 {\frac{1}{2}} 0 0 0 , \, \,
\T_{19} =  \chs {\frac{1}{2}} {\frac{1}{2}} 0 {\frac{1}{2}} {\frac{1}{2}} {\frac{1}{2}} , \, \,  \T_{20} =
\chs 0 {\frac{1}{2}} 0 0 0 {\frac{1}{2}} , \\
&  \T_{21} =  \chs 0 0 0 0 {\frac{1}{2}} 0 , \, \, \T_{22} =  \chs 0 {\frac{1}{2}} {\frac{1}{2}} 0 0 0, \, \,
\T_{23} =  \chs {\frac{1}{2}} {\frac{1}{2}} {\frac{1}{2}} {\frac{1}{2}} {\frac{1}{2}} 0 , \, \,
\T_{24} =  \chs {\frac{1}{2}} 0 {\frac{1}{2}} {\frac{1}{2}} 0 {\frac{1}{2}} \\
&\T_{25} =  \chs {\frac{1}{2}} 0 0 0 0 {\frac{1}{2}}, \, \,  \T_{26} =  \chs 0 0 0 {\frac{1}{2}}
{\frac{1}{2}} {\frac{1}{2}}, \, \,  \T_{27} =  \chs 0 {\frac{1}{2}} 0 {\frac{1}{2}} 0 0, \, \, \T_{28} = \chs
0 0 {\frac{1}{2}} {\frac{1}{2}} {\frac{1}{2}} 0,\\
&  \T_{29} =  \chs {\frac{1}{2}} 0 {\frac{1}{2}} 0 0 0 , \, \, \T_{30} =  \chs {\frac{1}{2}} {\frac{1}{2}}
{\frac{1}{2}} 0 {\frac{1}{2}} {\frac{1}{2}} , \, \, \T_{31} =  \chs{\frac{1}{2}} 0 {\frac{1}{2}} 0
{\frac{1}{2}} 0 , \, \,  \T_{32} =  \chs 0 0 {\frac{1}{2}} {\frac{1}{2}} 0 0 ,\\
&  \T_{33} =  \chs 0 {\frac{1}{2}} {\frac{1}{2}} {\frac{1}{2}} {\frac{1}{2}} {\frac{1}{2}} , \, \, \T_{34} =
\chs 0 0 0 {\frac{1}{2}} 0 {\frac{1}{2}} ,\, \,  \T_{35} =  \chs {\frac{1}{2}} 0 0 0
{\frac{1}{2}} {\frac{1}{2}} , \, \,  \T_{36} = \chs {\frac{1}{2}} {\frac{1}{2}} 0 0 0 0 \\
&\T_{37} =  \chs {\frac{1}{2}} 0 0 {\frac{1}{2}} 0 0  , \, \,  \T_{38} = \chs {\frac{1}{2}} {\frac{1}{2}} 0 0
{\frac{1}{2}} 0 , \, \,  \T_{39} =  \chs {\frac{1}{2}} {\frac{1}{2}} {\frac{1}{2}} 0 0 {\frac{1}{2}}, \, \,
\T_{40} = \chs 0 {\frac{1}{2}} 0 {\frac{1}{2}} {\frac{1}{2}} 0 ,\\
&  \T_{41} =  \chs 0 {\frac{1}{2}} {\frac{1}{2}} {\frac{1}{2}} 0 {\frac{1}{2}} , \, \,  \T_{42} = \chs 0 0
{\frac{1}{2}} 0 {\frac{1}{2}} {\frac{1}{2}}, \, \, \T_{43} =  \chs {\frac{1}{2}} {\frac{1}{2}} {\frac{1}{2}}
{\frac{1}{2}} 0 0 , \, \,  \T_{44} = \chs 0 {\frac{1}{2}} {\frac{1}{2}} 0 {\frac{1}{2}} 0 , \\
&  \T_{45} = \chs 0 0 {\frac{1}{2}} 0 0 {\frac{1}{2}} , \, \, \T_{46} =  \chs 0 {\frac{1}{2}} 0 0
{\frac{1}{2}} {\frac{1}{2}} ,\, \,  \T_{47} =  \chs {\frac{1}{2}} {\frac{1}{2}} 0 {\frac{1}{2}} 0
{\frac{1}{2}} ,\, \, \T_{48} =  \chs {\frac{1}{2}} 0 0 {\frac{1}{2}} {\frac{1}{2}} 0 \\
&\T_{49} =  \chs{\frac{1}{2}} 0 {\frac{1}{2}} {\frac{1}{2}} {\frac{1}{2}} 0, \, \,  \T_{50} = \chs
{\frac{1}{2}} 0 0 {\frac{1}{2}} 0 {\frac{1}{2}}, \, \,  \T_{51} =  \chs {\frac{1}{2}} {\frac{1}{2}} 0 0
{\frac{1}{2}} {\frac{1}{2}}, \, \, \T_{52} = \chs 0 0 {\frac{1}{2}} {\frac{1}{2}} {\frac{1}{2}}
{\frac{1}{2}}, \\
& \T_{53} = \chs 0 {\frac{1}{2}} {\frac{1}{2}} 0 0 {\frac{1}{2}}, \, \, \T_{54} = \chs 0 {\frac{1}{2}} 0 0
{\frac{1}{2}} 0 , \, \, \T_{55} =  \chs {\frac{1}{2}} 0 {\frac{1}{2}} 0 0 {\frac{1}{2}},  \, \,  \T_{56} =
\chs {\frac{1}{2}} {\frac{1}{2}} {\frac{1}{2}} {\frac{1}{2}} {\frac{1}{2}} {\frac{1}{2}},\\
& \T_{57} = \chs {\frac{1}{2}} {\frac{1}{2}} 0 {\frac{1}{2}} 0 0 , \, \, \T_{58} =  \chs {\frac{1}{2}}
{\frac{1}{2}} {\frac{1}{2}} 0 {\frac{1}{2}} 0, \, \,  \T_{59} = \chs {\frac{1}{2}} 0 {\frac{1}{2}}
{\frac{1}{2}} 0 0 , \, \, \T_{60} =  \chs {\frac{1}{2}} 0 0 {\frac{1}{2}} {\frac{1}{2}} {\frac{1}{2}} \\
&\T_{61} =  \chs {\frac{1}{2}} 0 {\frac{1}{2}} 0 {\frac{1}{2}} {\frac{1}{2}} , \, \,  \T_{62} = \chs 0 0
{\frac{1}{2}} {\frac{1}{2}} 0 {\frac{1}{2}}, \, \,   \T_{63} =  \chs 0 {\frac{1}{2}}
{\frac{1}{2}} {\frac{1}{2}} {\frac{1}{2}} 0 , \, \, \T_{64} =  \chs 0 {\frac{1}{2}} 0 {\frac{1}{2}} {\frac{1}{2}} {\frac{1}{2}} \\
\end{split}
\]
\end{small}
It can be shown that one of the corresponding even theta constants is zero. Let's pick $S =
\{1,2,3,4,5,6,7\}$ and $U = \{1,3,5,7\}.$ Let $T = U.$ Then, by Theorem ~\ref{vanishingProperty} the theta
constant corresponding to the characteristic $\e_T =\chs {\frac{1}{2}} {\frac{1}{2}} {\frac{1}{2}}
{\frac{1}{2}} 0 {\frac{1}{2}} $ is zero. That is $\T_{12}(0) = 0.$ Next, we  give the relation between theta
characteristics and branch points of the genus three hyperelliptic curve. Let $ B = \{a_1 , a_2 , a_3 , a_4 ,
a_5 , 1 , 0\}$ be the finite branch points of the curves and $U = \{a_1, a_3, a_5, 0\}$ be the set of odd
branch points.
\begin{lem}
Any genus 3 hyperelliptic curve is isomorphic to a curve given by the equation
\[Y^2=X(X-1)(X-a_1)(X-a_2)(X-a_3)(X-a_4)(X-a_5), \] where
\begin{small}
\[ a_1 =\frac{\T_{31}^2\T_{21}^2}{\T_{34}^2\T_{24}^2}, \, \,  a_2=
\frac{\T_{31}^2\T_{13}^2}{\T_{9}^2\T_{24}^2}, \, \,  a_3 = \frac{\T_{11}^2\T_{31}^2}{\T_{24}^2\T_{6}^2}, \,
\,   a_4 = \frac{\T_{21}^2\T_{7}^2}{\T_{15}^2\T_{34}^2}, \, \, a_5=
\frac{\T_{13}^2\T_{1}^2}{\T_{26}^2\T_{9}^2}.
\]
\end{small}
\end{lem}

\proof By using Lemma~\ref{Thomae} we have the following set of equation of theta constants and branch points
which are ordered  $a_1, a_2, a_3, a_4, a_5, 0, 1, \infty$. We use the notation $(i,j)$ for $(a_i-a_j)$.
\begin{small}
\[
\begin{split}
{\T_{{1}}}^{4} & = A \,\left(1,6 \right) \left(3,6 \right) \left(5,6 \right) \left( 1 ,3 \right)  \left( 1 ,5
\right) \left( 3, 5 \right) \left( 2 ,4
\right) \left( 2 ,7 \right)  \left( 4 ,7 \right) \\
{\T_{{2}}}^{4} & = A \, \left( 3,6 \right) \left( 5,6 \right) \left( 3 ,5 \right)  \left( 1 , 2 \right)
\left( 1 ,4
\right) \left(2 ,4 \right) \left( 3 , 7 \right)  \left( 5 ,7 \right) \\
{\T_{{3}}}^{4} & = A \, \left( 3,6 \right)\left( 4,6 \right) \left( 3 , 4 \right) \left( 1, 2 \right) \left(
1 , 5
\right) \left( 2 ,5 \right) \left( 1 ,7 \right) \left( 2,7 \right)  \left( 5 ,7 \right) \\
{\T_{{4}}}^{4} & = A \, \left( 2,6 \right)\left( 3,6 \right)\left(5,6  \right) \left( 2 ,3 \right) \left( 2
,5 \right)
\left( 3 ,5 \right) \left(1, 4 \right)  \left( 1 ,7 \right)  \left( 4 ,7 \right) \\
{\T_{{5}}}^{4} & = A \, \left(4,6  \right)\left( 5,6 \right) \left( 4,5 \right) \left( 1 ,2 \right) \left(1,
3 \right)
\left( 2 ,3\right) \left( 1 , 7 \right) \left( 2 ,7 \right)  \left( 3 ,7 \right) \\
{\T_{{6}}}^{4} & = A \, \left( 1,6 \right) \left( 2,6 \right) \left( 3 ,4 \right)  \left( 3 ,5 \right) \left(
4 ,5 \right) \left( 1 ,2
\right) \left( 1 , 7 \right)  \left( 2 ,7 \right) \\
{\T_{{7}}}^{4} & = A \, \left( 2,6 \right)\left( 3,6 \right)\left( 4,6 \right) \left( 1 ,5 \right) \left( 2
,3 \right) \left( 2 ,4 \right)  \left( 3 ,4
\right)  \left( 1 , 7 \right) \left( 5 ,7 \right) \\
{\T_{{8}}}^{4} & = A \, \left( 2,6 \right)\left( 3,6 \right) \left( 2 ,3 \right) \left( 1 , 4 \right) \left(
1 ,5 \right)  \left( 4 ,5
\right)  \left( 1 ,7 \right) \left( 4 , 7 \right) \left( 5 ,7 \right) \\
{\T_{{9}}}^{4} & = A \, \left( 1,6 \right)\left( 3,6 \right) \left( 1 ,3 \right)  \left( 2 ,4 \right) \left(
2 ,5 \right) \left( 4 ,5
\right) \left( 1, 7 \right) \left( 3 ,7 \right) \\
{\T_{{10}}}^{4} & = A \,\left( 3,6 \right) \left( 5,6 \right) \left( 3 ,5 \right)  \left( 1 ,2 \right) \left(
1, 4 \right) \left( 2 ,4
\right) \left( 1 ,7 \right)  \left( 2 ,7 \right) \left( 4 , 7 \right) \\
{\T_{{11}}}^{4} & = A \, \left(3,6  \right)\left( 4,6 \right) \left( 5,6 \right) \left( 3 ,4 \right) \left( 3
,5 \right)  \left( 4 ,5 \right) \left(
1 ,2 \right)  \left( 1, 7 \right)  \left( 2 ,7 \right) \\
{\T_{{13}}}^{4} & = A \, \left( 2,6 \right)\left(  4,6\right)\left(  5,6\right) \left( 1 ,3 \right) \left( 2
,4 \right) \left( 2 ,5 \right) \left( 4 ,5
\right) \left( 1 ,7 \right) \left( 3 ,7 \right) \\
{\T_{{14}}}^{4} & = A \, \left( 2,6 \right)\left( 5,6 \right) \left( 2 ,5 \right) \left( 1 ,3 \right) \left(
1 , 4 \right)  \left( 3 ,4
\right)  \left( 1 ,7 \right) \left( 3 ,7 \right) \left( 4 ,7 \right) \\
{\T_{{15}}}^{4} & = A \, \left( 1,6 \right) \left( 5,6 \right) \left( 1 ,5 \right) \left( 2 ,3 \right) \left(
2 ,4 \right) \left( 3 ,4
\right) \left( 1 , 7 \right) \left( 5 ,7 \right) \\
{\T_{{16}}}^{4} & = A \, \left( 1,6 \right) \left( 2 ,3 \right)  \left( 2 ,4 \right) \left( 2 ,5 \right)
\left( 3 ,4 \right) \left( 3 ,5 \right) \left( 4 , 5 \right) \left( 1 ,7 \right)
\\
{\T_{{17}}}^{4} & = A \, \left( 1,6 \right) \left( 4,6 \right) \left( 2 ,3 \right) \left( 2 ,5 \right) \left(
3 ,5 \right) \left( 1 , 4
\right) \left( 1 ,7 \right) \left( 4 ,7 \right) \\
{\T_{{18}}}^{4} & = A \, \left(2,6  \right) \left( 4,6 \right) \left( 1 ,3 \right) \left( 1 ,5 \right) \left(
3 ,5 \right) \left( 2 ,4
\right) \left( 1 ,7 \right) \left( 3 ,7 \right) \left( 5, 7 \right) \\
{\T_{{19}}}^{4} & = A \, \left( 3,6 \right)\left( 4,6 \right) \left( 1 ,2 \right) \left( 1 ,5 \right) \left(
2 ,5 \right) \left( 3 ,4
\right)  \left( 3 , 7 \right) \left( 4 , 7 \right) \\
{\T_{{20}}}^{4} & = A \, \left(  2,6\right) \left( 1 ,3 \right) \left( 1 , 4 \right) \left( 1 , 5 \right)
\left( 3 ,4 \right) \left( 3 ,5 \right) \left( 4 , 5 \right) \left( 2 , 7 \right)
\\
{\T_{{21}}}^{4} & = A \, \left( 1,6 \right)\left( 4,6 \right)\left( 5,6 \right) \left( 1 ,4 \right) \left( 1
,5 \right) \left( 4 ,5 \right) \left( 2 ,3
\right) \left( 2 , 7 \right) \left( 3 ,7 \right) \\
{\T_{{22}}}^{4} & = A \, \left( 1,6 \right)\left( 3,6 \right)\left( 4,6 \right) \left( 1 , 3 \right) \left( 1
, 4 \right) \left( 3 ,4 \right) \left( 2 ,5
\right)  \left( 2 , 7 \right) \left( 5 , 7 \right) \\
{\T_{{23}}}^{4} & = A \, \left( 1,6 \right)\left( 2,6 \right) \left(3 ,4 \right)  \left( 3 ,5 \right) \left(
4 , 5 \right) \left( 1 ,2
\right)  \left( 3 , 7 \right)  \left( 4, 7 \right) \left( 5, 7 \right) \\
{\T_{{24}}}^{4} & = A \, \left( 4,6 \right)\left( 5,6 \right) \left( 1 ,2 \right) \left( 1 ,3 \right) \left(
2, 3 \right) \left( 4 ,5
\right) \left( 4 , 7 \right) \left( 5 ,7 \right) \\
{\T_{{25}}}^{4} & = A \, \left( 3,6 \right) \left( 1 ,2 \right)  \left( 1 ,4 \right) \left( 1 ,5 \right)
\left( 2 ,4 \right) \left( 2 ,5 \right)  \left( 4 ,5 \right) \left( 3 ,7 \right) \\
{\T_{{26}}}^{4} & =  A \, \left( 2,6 \right)\left( 4,6 \right) \left( 1 ,3 \right) \left( 1 ,5 \right) \left(
3 ,5 \right) \left( 2 ,4
\right) \left( 2 ,7 \right) \left( 4 ,7 \right) \\
{\T_{{27}}}^{4} & = A \, \left( 1,6 \right)\left( 5,6 \right) \left( 1 ,5 \right)  \left( 2 ,3 \right) \left(
2 ,4 \right) \left( 3 ,4
\right)  \left( 2 ,7 \right)  \left( 3 ,7 \right) \left( 4 , 7 \right) \\
{\T_{{28}}}^{4} & = A \, \left( 1,6 \right)\left( 3,6 \right) \left( 1 ,3 \right) \left( 2 ,4 \right) \left(
2 ,5 \right) \left( 4 ,5
\right)  \left( 2 ,7 \right)  \left( 4 ,7 \right) \left( 5 ,7 \right) \\
{\T_{{29}}}^{4} & = A \, \left( 1,6 \right)\left( 2,6 \right)\left( 4,6 \right) \left( 3 ,5 \right) \left( 1
,2 \right) \left( 1 ,4 \right) \left( 2, 4
\right) \left( 3, 7 \right)  \left( 5 ,7 \right) \\
{\T_{{30}}}^{4} & = A \, \left(5,6  \right) \left( 1 ,2 \right) \left( 1,3 \right) \left( 1 ,4 \right) \left(
2, 3 \right) \left( 2 ,4 \right) \left( 3,4 \right) \left( 5 , 7 \right) \\
{\T_{{31}}}^{4} & = A \, \left( 1,6 \right)\left( 2,6 \right)\left(  3,6 \right) \left( 1, 2 \right) \left( 1
,3 \right) \left( 2 ,3 \right) \left( 4 ,5
\right) \left( 4 ,7 \right) \left( 5 ,7 \right) \\
\end{split}
\]

\[
\begin{split}
{\T_{{32}}}^{4} & = A \, \left( 1,6 \right)\left( 4,6 \right) \left( 2 ,3 \right) \left( 2, 5 \right) \left(
3 ,5 \right) \left( 1 , 4
\right) \left( 2 , 7 \right) \left( 3, 7 \right) \left( 5 , 7 \right) \\
{\T_{{33}}}^{4} & = A  \, \left( 2,6 \right)\left( 5,6 \right) \left( 1 ,3 \right) \left( 1 ,4 \right) \left(
3, 4 \right) \left( 2 ,5
\right) \left( 2 , 7 \right) \left( 5 ,7 \right) \\
{\T_{{34}}}^{4} & = A \, \left(  2,6\right)\left( 3,6 \right) \left( 1 ,4 \right)  \left( 1 , 5 \right)
\left( 4, 5 \right) \left( 2 ,3
\right) \left( 2 ,7 \right) \left( 3 , 7 \right) \\
{\T_{{35}}}^{4} & = A \, \left( 4,6 \right) \left( 1 ,2 \right) \left( 1 ,3 \right) \left( 1, 5 \right)
\left( 2, 3 \right) \left( 2 ,5 \right) \left( 3, 5 \right) \left( 4 ,7 \right)
\\
{\T_{{36}}}^{4} & = A \, \left(  1,6\right)\left( 2,6 \right)\left( 5,6 \right) \left( 1 ,2 \right) \left( 1
,5 \right) \left( 2 ,5 \right) \left( 3 ,4
\right) \left( 3, 7 \right) \left( 4 ,7 \right) \\
\end{split}
\]
\end{small}

\noindent By using the set of equations given above we have several choices for $a_1,\cdots,a_5$ in terms of
theta constants.
\begin{center}
\begin{tabular}{c c c c}
Branch Points &  \multicolumn {3}{c} {Possible Ratios}\\[5pt]
$a_1^2$ & $\left(\frac{\T_{36}^2\T_{22}^2}{\T_{33}^2\T_{19}^2}\right)^2$ &
$\left(\frac{\T_{31}^2\T_{21}^2}{\T_{34}^2\T_{24}^2}\right)^2$ &
$\left(\frac{\T_{29}^2\T_{1}^2}{\T_{26}^2\T_{2}^2}\right)^2$ \\
%
 $a_2^2$ & $\left(\frac{\T_{4}^2\T_{29}^2}{\T_{2}^2\T_{17}^2}\right)^2$ & $\left(\frac{\T_{36}^2\T_{7}^2}{\T_{15}^2\T_{19}^2}\right)^2$
&
$\left(\frac{\T_{31}^2\T_{13}^2}{\T_{9}^2\T_{24}^2}\right)^2$ \\
$a_3^2$ & $\left(\frac{\T_{4}^2\T_{22}^2}{\T_{33}^2\T_{17}^2}\right)^2$ &
$\left(\frac{\T_{11}^2\T_{31}^2}{\T_{24}^2\T_{6}^2}\right)^2$ &
$\left(\frac{\T_{7}^2\T_{1}^2}{\T_{26}^2\T_{15}^2}\right)^2$ \\
$a_4^2$ & $\left(\frac{\T_{11}^2\T_{29}^2}{\T_{2}^2\T_{6}^2}\right)^2$ &
$\left(\frac{\T_{21}^2\T_{7}^2}{\T_{15}^2\T_{34}^2}\right)^2$ &
$\left(\frac{\T_{22}^2\T_{13}^2}{\T_{9}^2\T_{33}^2}\right)^2$ \\
$a_5^2$ & $\left(\frac{\T_{4}^2\T_{21}^2}{\T_{34}^2\T_{17}^2}\right)^2$ &
$\left(\frac{\T_{11}^2\T_{36}^2}{\T_{19}^2\T_{6}^2}\right)^2$ &
$\left(\frac{\T_{13}^2\T_{1}^2}{\T_{26}^2\T_{9}^2}\right)^2$ \\
\end{tabular}
\end{center}
%
Let's select the following choices for $a_1, \cdots, a_5$.
\[
a_1 =  \frac{\T_{31}^2\T_{21}^2}{\T_{34}^2\T_{24}^2},  \, \,
 a_2=
\frac{\T_{31}^2\T_{13}^2}{\T_{9}^2\T_{24}^2}, \, \,  a_3 =\frac{\T_{11}^2\T_{31}^2}{\T_{24}^2\T_{6}^2},  \,
\, a_4 = \frac{\T_{21}^2\T_{7}^2}{\T_{15}^2\T_{34}^2},  \quad a_5=
\frac{\T_{13}^2\T_{1}^2}{\T_{26}^2\T_{9}^2}.
\]
This completes the proof.
\endproof

\begin{rem}
Unlike the genus 2 case, here only $\T_1,$ $ \T_6,$ $ \T_7,$ $\T_{11},$ $ \T_{15},$ $ \T_{24},$ $ \T_{31}$
are from one of the G\"opel groups.
\end{rem}

\subsubsection{Genus 3 non-hyperelliptic cyclic curves}
Using the Thomae's like formula for cyclic curves, for   each cyclic curve $y^n=f(x)$ one can express the roots of
$f(x)$ in terms of ratios of theta functions as in the hyperelliptic case. In this section we study such curves for
$g=3$. We only consider the families of curves with positive dimension since the curves which belong to 0-dimensional
families are well known. The proof of the following lemma can be found in \cite{RF}.
\begin{lem}\label{Shiga}
Let $f$ be a meromorphic function on $\X,$ and let $$(f) = \sum_{i=1}^{m}b_i - \sum_{i=1}^{m}c_i $$ be the divisor
defined by $f.$ Let's take paths from $P_0$ (initial point) to $b_i$ and $P_0$ to $c_i$ so that $\sum_{i=1}^{m}
\int_{P_0}^{b_i} \om = \sum_{i=1}^{m} \int_{P_0}^{c_i} \om .$

For an effective divisor $P_1 + \cdots + P_g$ we have
\begin{equation}
f(P_1) \cdots f(P_g) = \frac{1}{E} \prod_{k=1}  \frac{\T(\sum_{i} \int_{P_0}^{P_i} \om - \int_{P_0}^{b_k} \om -
\triangle , \tau )} {\T(\sum_{i} \int_{P_0}^{P_i} \om - \int_{P_0}^{c_k} \om - \triangle , \tau )}
\end{equation}
where $E$ is a constant independent of $P_1, \dots , P_g,$ the integrals from $P_0$ to $P_i$ take the same paths both
in the numerator and in the denominator, $\triangle$ denotes the Riemann's constant and $\int_{P_0}^{P_i} \om = \left(
\int_{P_0}^{P_i} \om_1, \dots, \int_{P_0}^{P_i} \om_g \right)^t.$
\end{lem}

Notice that the definition of thetanulls is different in this part from the definitions of the hyperelliptic
case. We define the following three theta constants.
\[
\T_1 = \T\chs{0}{\frac{1}{6}}{0}{\frac{2}{3}}{\frac{1}{6}}{\frac{2}{3}} \quad \T_2 = \T
\chs{0}{\frac{1}{6}}{0}{\frac{1}{3}}{\frac{1}{6}}{\frac{1}{3}} \quad \T_3 = \T
\chs{0}{\frac{1}{6}}{0}{0}{\frac{1}{6}}{0}
\]
Next we consider the cases 16, 18, 19 from Table~\ref{g3}.\\

\noindent \textbf{Case 18:} If the automorphism group is $C_3$ then the equation of $\X$ is given by
\[y^3=x(x-1)(x-s)(x-t).\]
Let $Q_i$ where $i= 1..5$ be ramifying points in the fiber of $0,1,s,t,\infty$ respectively. Consider the
meromorphic function $f = x$ on $\X$ of order 3. Then we have $(f) = 3 Q_1 - 3 Q_5.$ By applying the
Lemma~\ref{Shiga} with $P_0 = Q_5$ and an effective divisor $2Q_2 + Q_3$ we have the following.

\begin{equation} \label{Shiga1}
E s = \prod_{k=1}^3  \frac{\T( 2\int_{Q_5}^{Q_2} \om  + \int_{Q_5}^{Q_3} \om - \int_{Q_5}^{b_k} \om -
\triangle , \tau )} {\T( 2\int_{Q_5}^{Q_2} \om  + \int_{Q_5}^{Q_3} \om  - \triangle , \tau )}
\end{equation}
Again apply the Lemma~\ref{Shiga} with an effective divisor $Q_2 + 2Q_3$ we have the following.

\begin{equation} \label{Shiga2}
E s^2 = \prod_{k=1}^3  \frac{\T( \int_{Q_5}^{Q_2} \om  + 2\int_{Q_5}^{Q_3} \om - \int_{Q_5}^{b_k} \om -
\triangle , \tau )} {\T( \int_{Q_5}^{Q_2} \om  + 2\int_{Q_5}^{Q_3} \om  - \triangle , \tau )}
\end{equation}

By dividing Eq.~\eqref{Shiga2} by Eq.~\eqref{Shiga1} we have,
\begin{equation} \label{s}
\begin{split}
s =&  \prod_{k=1}^3  \frac{\T( \int_{Q_5}^{Q_2} \om  + 2\int_{Q_5}^{Q_3} \om - \int_{Q_5}^{b_k} \om -
\triangle
, \tau )} {\T( \int_{Q_5}^{Q_2} \om  + 2\int_{Q_5}^{Q_3} \om  - \triangle , \tau )}  \\
   & \times \prod_{k=1}^3  \frac {\T( 2\int_{Q_5}^{Q_2} \om  + \int_{Q_5}^{Q_3} \om  - \triangle , \tau )}{\T( 2\int_{Q_5}^{Q_2} \om
   + \int_{Q_5}^{Q_3} \om - \int_{Q_5}^{b_k} \om - \triangle
, \tau )}
\end{split}
\end{equation}
By a similar argument we have
\begin{equation} \label{t}
\begin{split}
t =&  \prod_{k=1}^3  \frac{\T( \int_{Q_5}^{Q_2} \om  + 2\int_{Q_5}^{Q_4} \om - \int_{Q_5}^{b_k} \om -
\triangle
, \tau )} {\T( \int_{Q_5}^{Q_2} \om  + 2\int_{Q_5}^{Q_4} \om  - \triangle , \tau )}  \\
   & \times \prod_{k=1}^3  \frac {\T( 2\int_{Q_5}^{Q_2} \om  + \int_{Q_5}^{Q_4} \om  - \triangle , \tau )}{\T( 2\int_{Q_5}^{Q_2} \om
   + \int_{Q_5}^{Q_4} \om - \int_{Q_5}^{b_k} \om - \triangle , \tau )}
\end{split}
\end{equation}
Computing the right hand side of Eq.~\eqref{s} and Eq.~\eqref{t} was the one of the main points of \cite{SH}.
Finally, we have
\[s = \frac{\T_2^3}{\T_1^3}, \textit{ and } \, \,  r = \frac{\T_3^3}{\T_1^3}.\]

\noindent \textbf{Case 19:} If the group is $C_6$ then the equation is $y^3=x(x-1)(x-s)(x-t)$ with $s = 1-t.$
By using results from  case 18, we have \[\T_2^3 = \T_1^3 - \T_3^3.\]

\noindent \textbf{Case 16:}  
In this case the equation of $\X$ is given by $$y^4=x(x-1)(x-t).$$ This curve has 4 ramifying points $Q_i$
where $i= 1..4$ in the fiber of $0,1,t,\infty$ respectively. Consider the meromorphic function $f = x$ on
$\X$ of order 4. Then we have $(f) = 4 Q_1 - 4 Q_4.$ By applying the Lemma~\ref{Shiga} with $P_0 = Q_4$ and
an effective divisor $2Q_2 + Q_3$ we have the following.

\begin{equation} \label{Shiga3}
E t = \prod_{k=1}^4  \frac{\T( 2\int_{Q_4}^{Q_2} \om  + \int_{Q_4}^{Q_3} \om - \int_{Q_4}^{b_k} \om -
\triangle , \tau )} {\T( 2\int_{Q_4}^{Q_2} \om  + \int_{Q_4}^{Q_3} \om  - \triangle , \tau )}
\end{equation}
Again apply the Lemma~\ref{Shiga} with an effective divisor $Q_2 + 2Q_3$ we have the following.

\begin{equation} \label{Shiga4}
E t^2 = \prod_{k=1}^4  \frac{\T( \int_{Q_4}^{Q_2} \om  + 2\int_{Q_4}^{Q_3} \om - \int_{Q_4}^{b_k} \om -
\triangle , \tau )} {\T( \int_{Q_4}^{Q_2} \om  + 2\int_{Q_4}^{Q_3} \om  - \triangle , \tau )}
\end{equation}

We have the following by dividing Eq.~\eqref{Shiga4} by Eq.~\eqref{Shiga3}
\begin{equation}\label{last}
\begin{split}
t =& \prod_{k=1}^4  \frac{\T( \int_{Q_4}^{Q_2} \om  + 2\int_{Q_4}^{Q_3} \om - \int_{Q_4}^{b_k} \om -
\triangle , \tau )} {\T( \int_{Q_4}^{Q_2} \om  + 2\int_{Q_4}^{Q_3} \om  - \triangle , \tau )}   \\
& \times \prod_{k=1}^4 \frac{\T( 2\int_{Q_4}^{Q_2} \om  + \int_{Q_4}^{Q_3} \om  - \triangle , \tau )}  {\T(
2\int_{Q_4}^{Q_2} \om  + \int_{Q_4}^{Q_3} \om - \int_{Q_4}^{b_k} \om - \triangle , \tau )}
\end{split}
\end{equation}

In order to compute the explicit formula for $t$ one has to find the integrals on the right hand side. Such
computations are long and tedious and we intend to include them in further work.

\begin{rem} In the case 16) of Table~\ref{g3}, the parameter $t$ is given by
\[
\T[e]^4 = A (t-1)^4 t^2,\]
where $[e]$ is the theta characteristics corresponding to the partition $(\{1\}, \{2\}, \{3\}, \{4\})$ and
$A$ is a constant; see \cite{NK} for details. However,   this is not satisfactory since we would like $t$ as
a rational function in terms of theta. The methods in \cite{NK} do not lead to another relation among $t$ and
the thetanulls since the only partition we could take  is the above.
\end{rem}

Summarizing all of the above we have:
\begin{lem}
Let $\X$ be a non-hyperelliptic genus 3 curve. The following are true:
\begin{description}
\item [i)] If $\Aut (\X) \iso C_3$, then $\X$ is isomorphic to a curve with equation
\[ y^3 = x (x-1) \left(x-\frac{\T_2^3}{\T_1^3}\right) \left(x- \frac{\T_3^3}{\T_1^3}\right).\]

\item [ii)] If $\Aut (\X) \iso C_6$, then $\X$ is isomorphic to a curve with equation
\[ y^3 = x (x-1) \left(x-\frac{\T_2^3}{\T_1^3}\right) \left(x- \frac{\T_3^3}{\T_1^3}\right) \,\,\textit{with}\,\,  \T_2^3 = \T_1^3 -
\T_3^3.\]
\item [iii)] If $\Aut (\X)$ is isomorphic to the group with GAP identity $(16, 13)$, then $\X$ is isomorphic to a curve with equation
\[y^4 = x (x-1)(x-t)\,\, \textit{with}\,\,\]
where $t$ is given by Eq.~\eqref{last}.
\end{description}
\end{lem}

It seems possible to generalize such techniques of computing the branch points in terms of the theta
functions for any cyclic cover of the projective line. We intend to pursue the ideas of these papers in
further work.

\medskip

\noindent \textbf{Acknowledgements:} The first ideas of this paper started during a visit of the second and
third author at Boston University during the Summer 2006. Both the second and third author want to thank
Prof. Previato for making that visit possible.

\vspace{.5cm}

\noindent Department of Mathematics and Statistics, \\
Boston University, \\
Boston,   MA 02215-2411\\
Email: ep@bu.edu

\vspace{.5cm}

\noindent Department of Mathematics and Statistics,\\
Oakland University, \\
546 Science and Eng. Building,\\
Rochester, MI 48309\\
Email: shaska@oakland.edu

\vspace{.5cm}

\noindent Department of Mathematics and Statistics,\\
Oakland University, \\
389 Science and Eng. Build.,\\
Rochester, MI 48309\\
Email: gswijesi@oakland.edu
\end{document}